\documentclass[ejs]{imsart}

\usepackage{dsfont}

\usepackage{amsfonts}
\usepackage{amssymb}
\usepackage{graphicx}
\usepackage{natbib}
\usepackage{amsmath}
\usepackage[a4paper]{geometry}

\newtheorem{theorem}{Theorem}

\newtheorem{corollary}{Corollary}

\newtheorem{definition}{Definition}

\newtheorem{lemma}{Lemma}

\newtheorem{remark}{Remark}

\newenvironment{proof}[1][Proof]{\noindent\textbf{#1.} }{\ \rule{0.5em}{0.5em}}

\def\prob{\mathbf P}
\def\esp{\mathbf E}

\def\R{\mathbb R}
\def\e{\varepsilon}
\def\dist{{\rm dist}}
\def\trace{\textrm{tr}}
\def\fcar{\mathds{1}}
\def\N{\mathcal{N}}

\def\b{\beta}

\begin{document}

\begin{frontmatter}

\title{Minimax rate of testing in sparse linear regression}
\runtitle{Minimax rate of testing in sparse linear regression}
\begin{aug}
\author{\fnms{Alexandra} \snm{Carpentier}\thanksref{m1}},
\author{\fnms{Olivier} \snm{Collier}\thanksref{m2,m4}},
\author{\fnms{La\"etitia} \snm{Comminges}\thanksref{m3,m4}},
\author{\fnms{Alexandre B.} \snm{Tsybakov}\thanksref{m4}},
\author{\fnms{Yuhao} \snm{Wang}\thanksref{m5}}
%\author{A. Carpentier, O. Collier,  L. Comminges, A.B. Tsybakov, Y. Wang}
%\author{\fnms{} \snm{} \thanksref{m2}}
%\and 
%\author{\fnms{Alexandre B.} \snm{Tsybakov} \thanksref{m3}}
%\affiliation{Modal'X, Universit\'e Paris-Ouest\thanksmark{m1}, Universit\'e Paris Dauphine\thanksmark{m2} and CREST, ENSAE\thanksmark{m3}}
\address{University of Magdeburg\thanksmark{m1}, Modal'X, Universit\'e Paris-Nanterre\thanksmark{m2}, CEREMADE, Universit\'e Paris-Dauphine\thanksmark{m3}, CREST, ENSAE\thanksmark{m4}, LIDS-IDSS, MIT\thanksmark{m5}}
\runauthor{Carpentier et al.}
\end{aug}
\date{}
\maketitle

\begin{abstract}
We consider the problem of testing the hypothesis that the parameter of linear regression model is 0 against 
an $s$-sparse alternative separated from 0 in the $\ell_2$-distance. We show that, in Gaussian linear regression model with $p<n$, where $p$ is the dimension of the parameter and $n$ is the sample size, the non-asymptotic minimax rate of testing has the form $\sqrt{(s/n)\log(1 + \sqrt{p}/s)}.$ We also show that this is the minimax rate of estimation of the $\ell_2$-norm of the regression parameter.
\end{abstract}



\end{frontmatter}

\section{Introduction}

%\subsection{The model}
This paper is deals with testing of hypotheses on the parameter of linear regression model under sparse alternatives. This problem has various applications in genetics, signal transmission and detection, and compressed sensing. A detailed description of these applications can be found, for example, in \cite{ACP2011}.  It is important to find optimal methods of testing in such a framework, and a natural approach is to define the notion of optimality in a minimax sense. The problem of  testing under sparse alternatives in a minimax framework was first studied by \cite{Ingster1997} and \cite{DonohoJin2004} who considered the Gaussian mean model. These papers were dealing with an asymptotic setting under the assumption that the sparsity index scales as a power of the dimension. Non-asymptotic setting for the Gaussian mean model was analyzed by \cite{Baraud2002} who established bounds on the minimax rate of testing up to a logarithmic factor.  Finally, the exact non-asymptotic minimax testing rate for the Gaussian mean model is derived in \cite{CCT2017}. In this paper, we present an extension of the results of \cite{CCT2017} to linear regression model with Gaussian noise. Note that the problem of minimax testing for linear regression under sparse alternatives was already studied in \cite{ITV2010}, \cite{ACP2011}, \cite{Verzelen2012}. Namely, \cite{ITV2010}, \cite{ACP2011} deal with an asymptotic setting under additional assumptions on the parameters of the problem while \cite{Verzelen2012} obtains non-asymptotic bounds up to a logarithmic factor in the spirit of \cite{Baraud2002}.   
Our aim here is to derive the non-asymptotic minimax rate of testing in Gaussian linear regression model with no specific assumptions on the parameters of the problem. We give a solution to this problem when $p<n$, where $p$ is the dimension and $n$ is the sample size. 

We consider the model
\begin{equation}\label{1}
Y=X\theta + \sigma \xi,
\end{equation}
where $\sigma>0$, $\xi \in \mathbb R^n$ is a vector of Gaussian white noise, i.e.,~$\xi \sim \mathcal N(0, I_n)$, $X$ is a $n \times p$ matrix with random entries, $I_n$ is the $n \times n$ identity matrix,  and $\theta \in \mathbb R^p$ is an unknown parameter. In what follows, we assume everywhere that $X$ is independent of $\xi$.

The following notation will be used below.
For $u=(u_1,\dots,u_p)\in \mathbb R^p$, we denote by $\|\cdot \|_2$ be the $\ell_2$-norm, i.e., 
$$
\|u\|_2^2 = \sum_{i =1}^{ p} |u_i|^2,
$$
and let $\|\cdot\|_0$ be the $\ell_0$ semi-norm, i.e.,
$$\|u\|_0 = \sum_{i=1}^p \fcar_{u_i \ne 0},$$
where $\fcar_{\{\cdot\}}$ is the indicator function. We denote by
$\langle u, v \rangle=u^Tv$ the inner product of $u \in \mathbb R^p, v \in \mathbb R^p$. We denote by $\lambda_{\min}(M)$ and by $\textrm{tr}[M]$ the minimal eigenvalue and the trace of matrix $M \in \mathbb R^{p\times p}$.
For an integer $s\in \{1,\dots, p\}$, we consider the set $B_0(s)$ of all $s$-sparse vectors in $\mathbb R^p$: 
$$B_0(s) : = \{ u \in \mathbb R^p : \|u\|_0 \leq s\}.$$

%\subsection{Problem setup}

Given the observations $(X,Y)$, we consider the problem of testing the hypothesis  
\begin{align}\label{tes:lb0}
H_0: \theta = 0 \quad \text{against \ the \ alternative} \quad H_1: \theta \in \Theta(s,\tau) 
\end{align}
where
$$
 \Theta(s,\tau) = \{\theta \in B_0(s) : \|\theta\|_2 \ge \tau\}
$$
for some $s\in \{1,\dots, p\}$ and $\tau>0$. Let $\Delta=\Delta(X,Y)$ be a statistic with values in $\{0,1\}$. 
We define the risk of test based on $\Delta$ as the sum of the first type error and the maximum second type error:
$$
\prob_0(\Delta=1) + \sup_{\theta\in \Theta(s,\tau)}\prob_{\theta}(\Delta=0),
$$
where $\prob_{\theta}$ denotes the joint distribution of $(X,Y)$ satisfying \eqref{1}.
The smallest possible value of this risk is equal to the minimax risk  
$$
{\mathcal R}_{s,\tau} : = \inf_\Delta \Big\{\prob_0(\Delta=1) + \sup_{\theta\in \Theta(s,\tau)}\prob_{\theta}(\Delta=0)\Big\}
$$
where $\inf_\Delta$ is the infimum over all $\{0,1\}$-valued statistics. 
We define the {\it minimax rate of testing on the class $B_0(s)$ with respect to the  $\ell_2$-distance} as a value $\lambda>0$, for which the following two properties hold:
\begin{itemize}
\item[(i)]  {\it (upper bound)} for any $\e\in (0,1)$ there exists $A_\e>0$ independent of $p,n,s,\sigma$ such that, for all $A>A_\e$,
\begin{equation}\label{test1}
{\mathcal R}_{s,A\lambda} \le \e,
\end{equation}
\item[(ii)] {\it (lower bound)}  for any $\e\in (0,1)$ there exists $a_\e>0$  independent of $p,n,s,\sigma$ such that, for all $0<A<a_\e$,
\begin{equation}\label{test2}
{\mathcal R}_{s,A\lambda} \ge 1-\e.
\end{equation}
\end{itemize}
Note that the rate $\lambda$ defined in this way is a non-asymptotic minimax rate of testing as opposed to the classical asymptotic definition that can be found, for example, in \cite{IngsterSuslina2003}.  It is shown in  \cite{CCT2017} that 
when $X$ is the identity matrix and $p=n$ (which corresponds to the Gaussian sequence model), the non-asymptotic minimax rate of testing on the class $B_0(s)$ with respect to the  $\ell_2$-distance has the following form:
\begin{equation}\label{3}
\lambda =\left\{
  \begin{array}{lcl}
 \sigma \sqrt{s\log (1+p/s^2)} & \text{if}& s<\sqrt{p},\\
  \sigma p^{1/4} & \text{if} & s\ge \sqrt{p}. 
   \end{array}
  \right.
\end{equation}
For the regression model with random $X$ and satisfying some strong assumptions, 
the asymptotic minimax rate of testing when $n, p$, and $s$ tend to $\infty$ such that 
$s=p^a$ for some $0<a<1$ is studied in \cite{ITV2010}.  In particular, it is shown in \cite{ITV2010} that for this configuration of parameters and if the matrix $X$ has i.i.d. standard normal entries, the asymptotic rate has the form
\begin{equation}\label{4}
%\min(\sqrt{k}, p^{1/4}) \min \Big(\sqrt{\frac{\log(2 + p/k^2)}{n}},  (\min(p,k^2)n)^{-1/4} \Big) = 
\lambda = \sigma \min \Big(\sqrt{\frac{s\log(p)}{n}},  n^{-1/4} , \frac{p^{1/4}}{\sqrt{n}} \Big).
\end{equation}
Similar result for a somewhat differently defined alternative $H_1$ is obtained in \cite{ACP2011}. 

Below we show that non-asymptotically, and with no specific restriction on the parameters $n,p$ and~$s$, the lower bound (ii) for the minimax rate of testing is valid with
\begin{equation}\label{5}
%\min(\sqrt{k}, p^{1/4}) \min \Big(\sqrt{\frac{\log(2 + p/k^2)}{n}},  (\min(p,k^2)n)^{-1/4} \Big) = 
\lambda = \sigma \min  \Big(\sqrt{\frac{s\log(2 + p/s^2)}{n}}, n^{-1/4} , \frac{p^{1/4}}{\sqrt{n}} \Big)
\end{equation}
whenever $X$ is a matrix with isotropic distribution and independent subgaussian rows (the definitions of subgaussian and isotropic distributions will be given in Section 3). Furthermore, we show that the matching upper bound holds when $X$ is a matrix with i.i.d. standard Gaussian entries and $p<n$. Note that for $p<n$ the expression \eqref{5} takes the form
\begin{equation}\label{5a}
%\min(\sqrt{k}, p^{1/4}) \min \Big(\sqrt{\frac{\log(2 + p/k^2)}{n}},  (\min(p,k^2)n)^{-1/4} \Big) = 
\lambda = \sigma \min  \Big(\sqrt{\frac{s\log(2 + p/s^2)}{n}}, \frac{p^{1/4}}{\sqrt{n}} \Big)
\end{equation}
It will be also useful to note that, since for $s\le \sqrt{p}$ the function $s\mapsto s\log(2 + p/s^2)$ is increasing and satisfies $\log(2 + p/s^2)\le 2\log(1 + p/s^2)$, the  rate \eqref{5a} can be equivalently (to within  an absolute constant factor) written as 
\begin{equation}\label{6}
\lambda =\left\{
  \begin{array}{lcl}
 \sigma \sqrt{\frac{s\log(1 + p/s^2)}{n}} & \text{if}& s<\sqrt{p},\\
  \sigma \frac{p^{1/4}}{\sqrt{n}} & \text{if} & s\ge \sqrt{p}. 
   \end{array}
  \right.
\end{equation}
This expression is analogous to  \eqref{3}. Finally, note that the rate can be written in  the following more compact form 
\begin{equation}
	\sigma \min  \Big(\sqrt{\frac{s\log(2 + p/s^2)}{n}}, \frac{p^{1/4}}{\sqrt{n}} \Big)  \asymp \sigma \sqrt{\frac{s\log(1 + \sqrt{p}/s)}{n}},
\end{equation}
where $\asymp$ denotes the equivalence up to an absolute constant factor.

\section{Upper bounds on the minimax rates}

In this section, we assume that $X$ is a matrix with i.i.d. standard Gaussian entries and $p < n$ and we establish an upper bound on the minimax rate of testing in the form \eqref{6}. This will be done by using a connexion between testing and estimation of functionals. We first introduce an estimator $\hat Q$ of the quadratic functional $\|\theta\|_2^2$ and establish an upper bound on its risk. Then, we deduce from this result an upper bound for the risk of the estimator 
$\hat{N}$ of the norm $\| \theta\|_2$ defined as follows:
\begin{align*}
\hat{N} = \sqrt{\max(\hat{Q}, 0)}.
\end{align*}
Finally, using $\hat{N}$ to define a test statistic we obtain an upper bound on the minimax rate of testing.

Introduce the notation
$$\alpha_s = \esp(Z^2 | Z^2 > 2 \log (1 + p / s^2))$$
where $Z$ is a standard normal random variable, and set
$$y_i=\{(X^T X)^{-1} X^T Y\}_i$$ 
where $\{(X^T X)^{-1} X^T Y\}_i$ is the $i$th component of the least squares estimator $(X^T X)^{-1} X^T Y$.
Note that the inverse $(X^T X)^{-1}$ exists almost surely since we assume in this section that $X$ is a matrix with i.i.d. standard Gaussian entries and $p< n$, so that $X$ is almost surely of full rank.
We consider the following estimator of the quadratic functional $\|\theta\|_2^2$:
\begin{eqnarray*}
\hat{Q} := \left\{ 
\begin{array}{lcl} \displaystyle
\sum_{i=1}^p y_{i}^2 - \sigma^2 \trace[(X^T X)^{-1}] & \textrm{if} & s \geq \sqrt{p}, \\ \displaystyle
\sum_{i=1}^p \left[ y_i^2 - \sigma^2 (X^T X)_{ii}^{-1} \alpha_s \right] \fcar_{y_i^2 > 2 \sigma^2 (X^T X)_{ii}^{-1} \log (1 + p / s^2)}& \textrm{if} & s < \sqrt{p}.
\end{array}
\right.
\end{eqnarray*}
Here and below $(X^T X)_{ij}^{-1}$ denotes the $(i,j)$th entry of matrix $(X^T X)^{-1}$. 

For any integers $n,p,s$ such that $s\le p$, set
\begin{align*}
\psi(s,p) = \left\{
  \begin{array}{lcl}
      \frac{s\log(1 + p/s^2)}{n} & \text{if}& s<\sqrt{p},\\
      \frac{p^{1/2}}{n} & \text{if} & s\geq \sqrt{p}.
   \end{array}
  \right.
\end{align*}

\begin{theorem}\label{thm:upper2}
Let $n,p,s$ be integers such that $s\le p, n\ge 9$, and $p \leq \min(\gamma n, n-8)$ for some constant $0 < \gamma < 1$. Let $r>0$, $\sigma>0$. Assume that all entries of matrix $X$ are i.i.d. standard Gaussian random variables.  Then there exists a constant $c>0$ depending only on $\gamma$ such that
\begin{align*}
\underset{\theta:\| \theta \|_0 \leq s, \| \theta \|_2 \leq r}{\sup} \esp_{\theta} [( \hat{Q} - \| \theta\|_2^2 )^2] \leq c\Big( \sigma^2 \frac{r^2}{n} +  \sigma^4 \psi^2(s,p)\Big).
\end{align*}
\end{theorem}

The proof of Theorem~\ref{thm:upper2} is given in Section~\ref{sec:proof_upper}.

Arguing exactly in the same way as in the proof of Theorem 8 in~\cite{CCT2017}, we deduce from Theorem~\ref{thm:upper2} the following upper bound on the squared risk of the estimator $\hat{N}$.

\begin{theorem}\label{thm:upper1}
Let the assumptions of Theorem~\ref{thm:upper2} be satisfied. Then there exists a constant $c'>0$ depending only on $\gamma$ such that
\begin{align*}
\underset{\theta \in B_0(s)}{\sup} \esp_{\theta} [( \hat{N} - \| \theta\|_2 )^2] \leq c' \sigma^2 \psi(s,p).
\end{align*}
\end{theorem}

Theorem~\ref{thm:upper1} implies that the test $\Delta^*= \fcar_{\{\hat{N}>A\lambda /2\}}$ where $\lambda = \sigma \sqrt{\psi(s,p)}$ (i.e., the same $\lambda$ as in \eqref{6}) satisfies
\begin{align*}\label{eq:tests}
&\prob_0(\Delta^*=1) + \sup_{\theta\in \Theta(s, A\lambda)}\prob_{\theta}(\Delta^*=0)\\
&\quad \le \prob_0(\hat N > A \lambda/2) + \sup_{\theta\in B_0(s)}\prob_{\theta}(\hat N - \|\theta\|_2 \le - A \lambda/2)
\nonumber\\
&\quad \le 2 \sup_{\theta\in B_0(s)} \frac{\esp_{\theta}[(\hat N- \|\theta\|_2)^2]}{(A/2)^2 \lambda^2}  \le C_*A^{-2}\nonumber
\end{align*}
for some constant $C_*>0$. Using this remark and choosing $A_\e=(C_*/\e)^{1/2}$  leads to the upper bound (i) that we have defined in the previous section. We state this conclusion in the next theorem.

\begin{theorem}\label{thm:upper3}
Let the assumptions of Theorem~\ref{thm:upper2} be satisfied and let $\lambda$ be defined by \eqref{6}. Then, for any $\e\in (0,1)$ there exists $A_\e>0$  depending only on $\e$ and $\gamma$ such that, for all $A>A_\e$,
$$
{\mathcal R}_{s,A\lambda} \le \e.
$$
\end{theorem}

\section{Lower bounds on the minimax rates}

In this section, we assume that the distribution of matrix $X$ is isotropic and has independent $\sigma_X$-subgaussian rows for some  $\sigma_X>0$. The isotropy of $X$ means that $E_X(X^TX/n)=I_p$ where $E_X$ denotes the expectation with respect to the distribution $P_X$ of $X$.  
\begin{definition}  Let $b>0$. A real-valued random variable $\zeta$ is called {$b$-subgaussian} if
$$
\esp \exp(t\zeta) \le \exp(b^2t^2/2), \quad \forall t\in \R. 
$$
A random vector $\eta$ with values in $\R^d$ is called {$b$-subgaussian} if all inner products $\langle\eta, v\rangle$ with  vectors $v\in \R^d$  such that $\|v\|_2=1$ are $b$-subgaussian random variables.
\end{definition}

The following theorem on the lower bound is non-asymptotic and holds with no restriction on the parameters $n,p,s$ except for the inevitable condition $s\le p$.

\begin{theorem}\label{th:lower}
Let $\e\in(0,1)$, $\sigma>0$, and let the integers $n,p,s$ be such that $s\le p$. Assume that the distribution of matrix $X$ is isotropic and $X$ has independent $\sigma_X$-subgaussian rows for some  $\sigma_X>0$. Then, there exists $a_\e>0$ depending only on $\e$ and $\sigma_X$ such that, for 
\begin{equation}\label{eq:th:lower}
%\min(\sqrt{k}, p^{1/4}) \min \Big(\sqrt{\frac{\log(2 + p/k^2)}{n}},  (\min(p,k^2)n)^{-1/4} \Big) = 
\tau = A \sigma  \min  \Big(\sqrt{\frac{s\log(2 + p/s^2)}{n}}, n^{-1/4} , \frac{p^{1/4}}{\sqrt{n}} \Big)
\end{equation}
with  any $A$ satisfying $0<A<a_\e$, we have
\begin{equation*}\label{}
{\mathcal R}_{s,\tau} \ge 1-\e.
\end{equation*}
\end{theorem}

The proof of Theorem~\ref{th:lower} is given in Section~\ref{sec:proof_lower}. The next corollary is an immediate consequence of Theorems~\ref{thm:upper3} and~\ref{th:lower}.

\begin{corollary}\label{cor1}
Let the assumptions of Theorem~\ref{thm:upper2} be satisfied. Then the minimax rate of testing on the class $B_0(s)$ with respect to the  $\ell_2$-distance is given by~\eqref{5a}. 
 \end{corollary}

In addition, from Theorem~\ref{th:lower}, we get the following lower bound on the minimax risk of estimation of the $\ell_2$-norm $\| \theta\|_2$.

\begin{theorem}\label{th:lower1}
Let the assumptions of Theorem~\ref{th:lower} be satisfied, and let $\lambda$ be defined in \eqref{5}. Then there exists an a constant $c_*>0$ depending only on $\sigma_X$ such that
\begin{align*}
\inf_{\hat{T}}\underset{\theta \in B_0(s)}{\sup} \esp_{\theta} [ (\hat{T} - \| \theta\|_2 )^2] \geq c_* \lambda^2,
\end{align*}
where $\inf_{\hat{T}}$ denotes the infimum over all estimators.
\end{theorem}
The result of Theorem~\ref{th:lower1} follows from Theorem~\ref{th:lower} by noticing that, for $\tau$ in \eqref{eq:th:lower} and $\lambda$ in \eqref{5} we have $\tau=A\lambda$, and for any estimator   $ \hat{T}$,
\begin{align*}%\label{eq:norm}
\underset{\theta \in B_0(s)}{\sup} \esp_{\theta} [( \hat{T} - \| \theta\|_2 )^2]
&\ge \frac12\Big[\esp_{0} [\hat{T}^2] + \sup_{\theta\in \Theta(s, \tau)} \esp_{\theta} [( \hat{T} - \| \theta\|_2 )^2] \Big]
\\
&\ge \frac{\tau^2}{8}\Big[\prob_{0} (\hat{T}>\tau/2) + \sup_{\theta\in \Theta(s, \tau)} \prob_{\theta} ( \hat{T} \le \tau/2) \Big]\\
&\ge \frac{(A\lambda)^2}{8} {\mathcal R}_{s,\tau}\,.
\end{align*}

\begin{corollary}\label{cor2}
Let the assumptions of Theorem~\ref{thm:upper2} be satisfied and let $\lambda$  be defined in~\eqref{5a}. Then the minimax rate of estimation of the norm $\| \theta\|_2$ under the mean squared risk on the class $B_0(s)$ is equal to $\lambda^2$, that is 
\begin{align*}
c_*\lambda^2\le \inf_{\hat{T}}\underset{\theta \in B_0(s)}{\sup} \esp_{\theta} [( \hat{T} - \| \theta\|_2 )^2 ]\le c'\lambda^2,
\end{align*} 
where $c_*>0$ is an absolute constant and $c'>0$ is a constant depending only on $\gamma$.
 \end{corollary}

This corollary is an immediate consequence of Theorems~\ref{thm:upper1} and~\ref{th:lower1}.

\begin{remark}
Inspection of the proofs in the Appendix reveals that the results of this section remain valid if we replace the $\ell_0$-ball $B_0(s)$ by the $\ell_0$-sphere
$\bar B_0(s)  = \{ u \in \mathbb R^p : \|u\|_0 = s\}.$
\end{remark}

%%%%%%%%%%%%%%%%%%%%
\section{APPENDIX}

\subsection{Preliminary lemmas for the proof of Theorem~\ref{thm:upper2}} \label{sec:upproof}

This section treats two main technical issues for the proof of Theorem~\ref{thm:upper2}. The first one is to control the expectation of a power of the smallest eigenvalue of the inverse empirical covariance matrix. The second issue is to control the errors for identifying non-zero entries in the sparse setting. For this, we need accurate bounds on the correlations between centred thresholded transformations of two correlated $\chi^2_1$ random variables. We first recall two general facts that we will use to solve the first issue.  

In what follows, we will denote by $C$ positive constants that can vary from line to line.

\begin{lemma} \label{lem:upper0_1}[\cite{DS2001}, see also~\cite{versh}.]
Let $X$ satisfy the assumptions of Theorem \ref{thm:upper2}.
Let $\lambda_{\min}(\hat{\Sigma})$ denote the smallest eigenvalue of the sample covariance matrix $\hat{\Sigma} = \frac{1}{n} X^T X$. Then for any $t>0$ with probability at least $1-2\exp(-t^2/2)$ we have
\begin{align*}
1-\sqrt{\frac{p}{n}}-\frac{t}{\sqrt{n}}\le \sqrt{\lambda_{\min}(\hat{\Sigma})}  \leq 1+\sqrt{\frac{p}{n}}+\frac{t}{\sqrt{n}}.
\end{align*}

\end{lemma}

\begin{lemma} \label{lem:upper0_2}[\cite[Lemma A4]{TV2010}, see also \cite[Lemma 4.14]{BC2012}.]
Let $1 \leq p \leq n$, let $R_i$ be the $i$-th column of matrix $X\in \R^{n \times p}$ and $R_{-i} = {\rm span}\{R_j : j \neq i\}$. If $X $ has full rank, then 
\begin{align*}
(X^T X)^{-1}_{ii} = \dist(R_i, R_{-i})^{-2},
\end{align*}
where $\dist(R_i, R_{-i})$ is the Euclidean distance of vector $R_i$ to the space $R_{-i}$.
\end{lemma}

\begin{lemma}\label{lem:upper1}
Let $n\ge 9$ and $p \leq \min(\gamma n, n-8)$ for some constant $\gamma$ such that $0 < \gamma < 1$. Assume that all entries of matrix $X\in \R^{n \times p}$ are i.i.d. standard Gaussian random variables. Then there exists a constant $c>0$ depending only on $\gamma$, such that
\begin{align} \label{eq:upper1_2}
 \esp [\lambda_{\min}^{-2}(\hat{\Sigma})] \leq c.
\end{align}
\end{lemma}
\begin{proof} Set $\beta=\sqrt{\gamma}$. From the inequality $p \leq \gamma n$ and Lemma~\ref{lem:upper0_1} we have
\begin{align*}
\prob\Big(\sqrt{\lambda_{\min}(\hat{\Sigma})}< 1-\beta - \frac{t}{\sqrt{n}}\Big) &\le \prob\Big(\sqrt{\lambda_{\min}(\hat{\Sigma})}< 1-\sqrt{\frac{p}{n}} - \frac{t}{\sqrt{n}}\Big) 
\\
 &\le 2\exp(-t^2/2).
\end{align*}
Taking here
$t= \sqrt{n}(1-\beta)/2$
we arrive at the inequality
$$\prob\Big(\lambda_{\min}(\hat{\Sigma})< \Big(\frac{1-\b}{2}\Big)^2\Big)\le 2\exp\Big(-\frac{n(1-\beta)^2}{8}\Big).$$
Using this inequality we obtain
\begin{equation}\label{lem:upper1_1}
\esp[ \lambda_{\min}^{-2}(\hat{\Sigma})]\le  \Big(\frac{1-\b}{2}\Big)^{-4}+\sqrt{ \esp[ \lambda_{\min}^{-4}(\hat{\Sigma})] }\sqrt{2}\exp\Big(-\frac{n(1-\beta)^2}{16}\Big). \end{equation}
We now bound the expectation $ \esp[ \lambda_{\min}^{-4}(\hat{\Sigma})]$. Clearly, 
\begin{equation}\label{lem:upper1_2} \lambda_{\min}^{-1}(\hat{\Sigma})\le\textrm{tr}[\hat{\Sigma}^{-1}]. \end{equation} 
Lemma~\ref{lem:upper0_2} implies that, almost surely, 
\begin{equation*}\big(\textrm{tr}[\hat{\Sigma}^{-1}]\big)^4 = n^4 \big[\sum_{i=1}^p \dist(R_i, R_{-i})^{-2}\big]^4 \le n^4 p^3 \sum_{i=1}^p \dist(R_i, R_{-i})^{-8}.
\end{equation*}
Since the random variables $\dist(R_i, R_{-i})$ are identically distributed and $p\le n$ we have 
\begin{equation}\label{lem:upper1_3}
\esp \big[\big(\textrm{tr}[\hat{\Sigma}^{-1}]\big)^4\big] \le n^8 \esp [\dist(R_1, R_{-1})^{-8}].
\end{equation}
Finally we only need to  bound $\esp [\dist(R_1, R_{-1})^{-8}]$. If $\mathcal{S}$ is a $p-1$ dimensional subspace of $\R^n$ then the random variable $\dist(R_1, \mathcal{S})^2$ has the chi-square distribution $\chi^2_{n-p+1}$ with $n  - p + 1$ degrees of freedom.  Hence, as  $R_{-1}$ is a span of random vectors independent of $R_1$ and $R_{-1}$ is almost surely $p-1$ dimensional, we have 
\begin{align}\nonumber
\esp [\dist(R_1, R_{-1})^{-8}]&=  \esp\bigg[\frac{1}{(\chi^2_{n-p+1})^4}\bigg]\\
&= \frac{1}{(n-p-1)(n-p-3)(n-p-5)(n-p-7)}\le \frac{1}{105} . \label{lem:upper1_4}
\end{align}
Combining \eqref{lem:upper1_1}, \eqref{lem:upper1_2}, \eqref{lem:upper1_3} and \eqref{lem:upper1_4} we get 
$$\esp[ \lambda_{\min}^{-2}(\hat{\Sigma})]\le  \Big(\frac{1-\b}{2}\Big)^{-4}+\frac{n^8}{105}\sqrt{2}\exp\Big(-\frac{n(1-\beta)^2}{16}\Big),$$
which implies the lemma.
\end{proof}

We now turn to the second issue of this section, that is bounds on the correlations. We will use the following lemma about the tails of the standard normal distribution.   

\begin{lemma}\label{lem:upper2_0}
For $\eta \sim \N(0,1)$ and any $x > 0$ we have
\begin{align}\label{lem:upper2_0_1}
\frac{4}{\sqrt{2\pi} (x + \sqrt{x^2 + 4})} \exp(-x^2 / 2) \leq \prob (| \eta | > x) \leq \frac{4}{\sqrt{2\pi} (x + \sqrt{x^2 + 2})} \exp(-x^2 / 2),
\end{align}
\begin{align}\label{lem:upper2_0_2}
\esp [\eta^2 \fcar_{|\eta| > x}] \leq \sqrt{\frac{2}{\pi}} \left(x + \frac{2}{x}\right) \exp(-x^2 / 2),
\end{align}
\begin{align}\label{lem:upper2_0_3}
\esp [\eta^4 \fcar_{|\eta| > x}] \leq \sqrt{\frac{2}{\pi}} \left(x^3 + 3x + \frac{1}{x}\right) \exp(-x^2 / 2).
\end{align}
Moreover, if $x\ge 1$, then 
\begin{align}\label{lem:upper2_0_4}
x^2< \esp[\eta ^2 \mid |\eta | > x] \le 5x^2.
\end{align}
\end{lemma}

Inequalities \eqref{lem:upper2_0_1} - \eqref{lem:upper2_0_3} are given, e.g., in \cite[Lemma~4]{CCT2017} and \eqref{lem:upper2_0_4} follows easily from \eqref{lem:upper2_0_1} and \eqref{lem:upper2_0_2}.

%\begin{lemma}\label{lem:upper2}
%Let  $(\eta,\zeta)$ be a Gaussian vector with mean 0 and covariance matrix $\Gamma=\begin{pmatrix} 1&\rho\\\rho&1\end{pmatrix}$, $0<\rho<1$. Set $\alpha = \esp[\eta ^2 \mid | \eta | > x]$.
%Then there exists an absolute constant $C>0$ such that, for any $x \ge 1$, 
%\begin{align*}
%\esp[(\eta^2 - \alpha) (\zeta^2 - \alpha) \fcar_{\vert \eta \vert > x} \fcar_{\vert \zeta \vert > x}] \leq C x^3 \exp(-x^2 / 2).
%\end{align*}
%\end{lemma}
%
%\begin{proof}
%\end{proof}

\begin{lemma}\label{lem:upper3}
Let  $(\eta,\zeta)$ be a Gaussian vector with mean 0 and covariance matrix $\Gamma=\begin{pmatrix} 1&\rho\\\rho&1\end{pmatrix}$, $0<\rho<1$. Set $\alpha = \esp[\eta ^2 \mid | \eta | > x]$.
Then there exists an absolute constant $C>0$ such that, for any $x \ge 1$, 
\begin{align*}
\esp[(\eta^2 - \alpha) (\zeta^2 - \alpha) \fcar_{\vert \eta \vert > x} \fcar_{\vert \zeta \vert > x}] \leq C \rho^2 x^4 \exp(-x^2 / 2).
\end{align*}
\end{lemma}

\begin{proof}
From \eqref{lem:upper2_0_4} we get that $\alpha \leq 5 x^2$. Thus,  using \eqref{lem:upper2_0_3} and the fact that $x \ge 1$ we find
\begin{align} \label{eq:upper2_0}
%\begin{split}
\esp[(\zeta^2 - \alpha)^2  \fcar_{\vert \zeta \vert > x}] &\leq  
\esp \left[(\zeta^4+\alpha^2) \fcar_{\vert \zeta \vert > x} \right] 
\le
26 \esp[\zeta^4 \fcar_{\vert \zeta \vert > x}]
\leq C x^3 \exp\left( -x^2/2\right). 
%\end{split}
\end{align}
Therefore,
\begin{align*}
\begin{split}
\esp[(\eta^2 - \alpha) (\zeta^2 - \alpha) \fcar_{\vert \eta \vert > x} \fcar_{\vert \zeta \vert > x}]& 
  \leq 
 \esp[(\eta^2 - \alpha)^2  \fcar_{\vert \eta \vert > x}] + \esp[(\zeta^2 - \alpha)^2  \fcar_{\vert \zeta \vert > x}]
 \leq C x^3 \exp\left( -x^2/2\right).
\end{split}
\end{align*}
This proves the lemma for $\rho \ge1/\sqrt{5}$.

Consider now the case $0 < \rho < 1/\sqrt{5}$. 
Note that, since $\alpha \leq 5 x^2$, for $0 < \rho < 1/\sqrt{5}$ we also have
\begin{align*}
\rho  <  \frac{x}{\sqrt{\alpha}}.
\end{align*}
The symmetry of the distribution of $(\eta,\zeta)$ implies
\begin{align}~\label{eq:upper3_1}\begin{split}
\esp[(\eta^2 - \alpha) (\zeta^2 - \alpha) \fcar_{\vert \eta \vert > x} \fcar_{\vert \zeta \vert > x}]& =2\esp[(\eta^2 - \alpha) (\zeta^2 - \alpha) \fcar_{\vert \eta \vert > x} \fcar_{ \zeta  > x}].
\end{split}
\end{align}
Now, we use the fact that 
%\begin{align*} %\label{eq:A}
 $(\eta,\zeta) \overset{d}{=} (\rho \zeta + \sqrt{1-\rho^2}Z, \zeta)$
%\end{align*}
where $\overset{d}{=}$ means equality in distribution and $Z$ is a standard Gaussian random variable independent of $\zeta$. 
Thus,  
\begin{align}~\label{eq:upper3_2}\begin{split}
\esp[(\eta^2 - \alpha) (\zeta^2 - \alpha) \fcar_{\vert \eta \vert > x} \fcar_{ \zeta  > x}] & = \rho^2 \esp[(\zeta^2 - \alpha)^2 \fcar_{\lvert \rho \zeta+\sqrt{1-\rho^2}Z\rvert  > x} \fcar_{ \zeta  > x}] \\
& \qquad + 2 \rho \sqrt{1 - \rho^2} \esp[\zeta Z(\zeta^2 - \alpha) \fcar_{ \vert \rho \zeta+\sqrt{1-\rho^2}Z \vert  > x} \fcar_{ \zeta > x}] \\
& \qquad + (1 - \rho^2) \esp[(Z^2 - \alpha) (\zeta^2 - \alpha) \fcar_{ \vert \rho \zeta+\sqrt{1-\rho^2}Z \vert  > x} \fcar_{ \zeta > x}].
\end{split}
\end{align}
We now bound separately the three summands on the right hand side of \eqref{eq:upper3_2}. For the first summand, using \eqref{eq:upper2_0} %{lem:upper2_0_3}, the inequality $\alpha \leq 5x^2$, and the fact that $x \ge 1$ 
we get the bound
\begin{equation}\label{eq:upper3_3}
\rho^2 \esp[(\zeta^2 - \alpha)^2 \fcar_{ \vert \rho \zeta+\sqrt{1-\rho^2}Z \vert  > x} \fcar_{ \zeta > x}] \leq 26 \rho^2 \esp[\zeta^4 \fcar_{ \zeta  > x}] \leq C \rho^2 x^3 \exp\left( -\frac{x^2}{2}\right).
\end{equation}
To bound the second summand, we first write
\begin{equation}\label{prem}
\esp[\zeta Z(\zeta^2 - \alpha) \fcar_{ \vert \rho \zeta+\sqrt{1-\rho^2}Z\vert > x} \fcar_{ \zeta  > x}] = \esp[\zeta(\zeta^2-\alpha)  g(\zeta) \fcar_{ \zeta > x}]
\end{equation}
where  $g(\zeta) := \esp[Z \fcar_{ \lvert \rho \zeta+\sqrt{1-\rho^2}Z\rvert > x}  \mid \zeta]$. It is straightforward to check that 
\begin{align*}
%\begin{split}
g(\zeta) &  %=\int_{-\infty}^{-\frac{x + \rho Y}{\sqrt{1 - \rho^2}}} z \exp\left( - \frac{z^2}{2}\right) dz + \int_{\frac{x - \rho Y}{\sqrt{1 - \rho^2}}}^{\infty} z \exp\left( - \frac{z^2}{2}\right) dz \\
 = \exp\left( - \frac{(x - \rho \zeta)^2}{2(1 - \rho^2)} \right) - \exp\left( - \frac{(x + \rho \zeta)^2}{2(1 - \rho^2)} \right).
%\end{split}
\end{align*}
Thus $g(\zeta)$ is positive when $\zeta>x$. Therefore we have
\begin{equation}\label{eq:upper3_4}
 \esp[\zeta (\zeta^2 - \alpha)g(\zeta) \fcar_{\zeta > x}] \leq  \esp[\zeta^3 g(\zeta) \fcar_{\zeta > x}].
\end{equation}
In addition, 
\begin{equation}\label{eq:upper3_5}
g(\zeta) =  \exp\left( - \frac{(x - \rho \zeta)^2}{2(1 - \rho^2)} \right) \left(1 - \exp\left( - \frac{2x\rho \zeta}{1 - \rho^2} \right)\right) \leq 1 - \exp\left( - \frac{2x\rho \zeta}{1 - \rho^2} \right) \leq \frac{2x\rho \zeta}{1 - \rho^2}.
\end{equation}
Combining \eqref{prem} - \eqref{eq:upper3_5} with~\eqref{lem:upper2_0_3} and the fact that $\rho \leq \frac{1}{2}$, we get 
\begin{equation}\label{secsum}
2 \rho \sqrt{1 - \rho^2}\esp[\zeta Z(\zeta^2 - \alpha) \fcar_{\vert \rho \zeta+\sqrt{1-\rho^2}Z \vert > x} \fcar_{\zeta  > x}]   \leq C \rho^2 x^4 \exp\left( -\frac{x^2}{2}\right).
\end{equation}

We now consider the third summand on the right hand side of~\eqref{eq:upper3_2}. We will prove that 
\begin{equation}\label{thirdsum}
\esp[(Z^2 - \alpha) (\zeta^2 - \alpha) \fcar_{\vert \rho \zeta+\sqrt{1-\rho^2}Z \vert > x} \fcar_{ \zeta  > x}]\le 0.
\end{equation}
We have 
\begin{align*}
\esp[(Z^2 - \alpha) (\zeta^2 - \alpha) \fcar_{\vert \rho \zeta+\sqrt{1-\rho^2}Z \vert > x} \fcar_{ \zeta  > x}] = 
\esp[(\zeta^2 - \alpha) f(\zeta) \fcar_{\zeta  > x} ]
\end{align*}
where  
\begin{align*}
f(\zeta)& := \esp[(Z^2 - \alpha)  \fcar_{\vert \rho \zeta+\sqrt{1-\rho^2}Z \vert > x} \mid \zeta]\\
&=\int_{\frac{x - \rho \zeta}{\sqrt{1 - \rho^2}}}^\infty (z^2 - \alpha) \exp \left(-\frac{z^2}{2}\right) dz + \int_{-\infty}^{-\frac{x + \rho \zeta}{\sqrt{1 - \rho^2}}} (z^2 - \alpha) \exp \left(-\frac{z^2}{2}\right) dz. \\
\end{align*} 
Note that $x< \sqrt{\alpha}$ by \eqref{lem:upper2_0_4}. In order to prove \eqref{thirdsum}, it is enough to show  that
\begin{align}\label{eq:upper3_6}
\forall\, \zeta \in [x, \sqrt{\alpha}],  \qquad f(\zeta)\ge f(\sqrt{\alpha}) .
\end{align}
and 
\begin{align}\label{eq:upper3_6bis}\forall\, \zeta \in [\sqrt{\alpha}, \infty),   \qquad   f(\zeta)\le f(\sqrt{\alpha}) .
\end{align}
Indeed, assume that \eqref{eq:upper3_6} and   \eqref{eq:upper3_6bis}  hold. Then we have 
\begin{align*}
\begin{split}
& \esp[(\zeta^2 - \alpha) f(\zeta) \fcar_{x <  \zeta  \leq \sqrt{\alpha}}] \leq \esp[(\zeta^2 - \alpha) f(\sqrt{\alpha}) \fcar_{x <  \zeta  \leq \sqrt{\alpha}}] \\
&\quad = - \esp[(\zeta^2 - \alpha) f(\sqrt{\alpha}) \fcar_{\zeta  > \sqrt{\alpha}}] \leq - \esp[(\zeta^2 - \alpha) f(\zeta) \fcar_{ \zeta  > \sqrt{\alpha}}],
\end{split}
\end{align*}
where the equality is due the fact that, by the symmetry of the normal distribution and the definition of $\alpha$,  
$$
\esp[(\zeta^2 - \alpha) \fcar_{\zeta  >x}]=\frac12 \esp[(\zeta^2 - \alpha) \fcar_{\vert\zeta\vert  >x}]=0.
$$

Thus, to finish the proof of the lemma, it remains to prove~\eqref{eq:upper3_6} and \eqref{eq:upper3_6bis}. 
We first establish~\eqref{eq:upper3_6}, for which it is sufficient to show that $f'(\zeta)< 0$ for $\zeta\in [x,\sqrt{\alpha}]$.  
Since $0<\rho <    x / \sqrt{\alpha}$ and $x<\sqrt{\alpha}$, we have 
\begin{equation}\label{eq:upper3_7}
\frac{(x - \rho y)^2}{1 - \rho^2} <  \alpha \quad \text{for all}\quad y\in [x,\sqrt{\alpha}].
\end{equation}
Using \eqref{eq:upper3_7} we obtain, for all $\zeta \in [x, \sqrt{\alpha}]$,
\begin{align*}
\begin{split}
f'(\zeta) & =\frac{\rho}{\sqrt{1 - \rho^2}}\exp\left( - \frac12 \Big(\frac{x + \rho \zeta}{\sqrt{1 - \rho^2}}\Big)^2  \right) \left( \left(\frac{(x - \rho \zeta)^2}{1 - \rho^2} - \alpha\right)\exp\left(\frac{2 \rho x \zeta}{1 - \rho^2}\right) - \left(\frac{(x + \rho \zeta)^2}{1 - \rho^2} - \alpha\right)\right)\\
&\leq \frac{\rho}{\sqrt{1 - \rho^2}} \exp\left( - \frac12\Big(\frac{x + \rho \zeta}{\sqrt{1 - \rho^2}}\Big)^2  \right) \left( \left(\frac{(x - \rho \zeta)^2}{1 - \rho^2} - \alpha\right) - \left(\frac{(x + \rho \zeta)^2}{1 - \rho^2} - \alpha\right)\right) \\
& = - \frac{\rho}{\sqrt{1 - \rho^2}} \exp\left( - \frac12\Big(\frac{x + \rho \zeta}{\sqrt{1 - \rho^2}}\Big)^2  \right) \frac{4x\rho \zeta}{1 - \rho^2} < 0.
\end{split}
\end{align*}
This implies~\eqref{eq:upper3_6}. Finally, we prove \eqref{eq:upper3_6bis}. To do this, it is enough to establish the following three facts:
\begin{itemize}
\item[(i)] $f'$ is continuous and $f'(\sqrt{\alpha})<0$;
\item[(ii)] the equation   $ f'(y) = 0$ has at most one solution on $[\sqrt{\alpha}, +\infty)$; 
\item[(iii)]  $ f(\infty)=\lim_{y\to \infty}f(y)\le f(\sqrt{\alpha})$. 
\end{itemize}
Property (i) is already proved above. To prove (ii), we first observe that the solution of the equation $\frac{d}{dy} f(y) = 0$ is also solution of the equation $h(y)=0$ where
\begin{align*}
h(y):=\left(\frac{(x - \rho y)^2}{1 - \rho^2} - \alpha\right) \left( \exp\left( \frac{2\rho x y}{1 - \rho^2}\right) - 1\right) - \frac{4\rho x y}{1 - \rho^2}.
\end{align*} 
Next,  let $y_1$ and $y_2$  be the solutions of the quadratic equation $\frac{(x - \rho y)^2}{1 - \rho^2} = \alpha$ :
$$y_1= \frac{x-\sqrt{\alpha(1-\rho^2)}}{\rho }\quad \text{ and } y_2=\frac{x+\sqrt{\alpha(1-\rho^2)}}{\rho }.$$
Due to \eqref{eq:upper3_7} we have  $y_1<\sqrt{\alpha}<y_2$. Thus, $h(y)< 0$ on the interval $[\sqrt{\alpha}, y_2]$. Next, on the interval  $(y_2,+\infty)$  the function $h$ is  strictly convex  and $h(y_2)<0$. It follows that $h$ vanishes only once on $(y_2,+\infty)$.
Thus, (ii) is proved.

It remains to show that   $f(\sqrt{\alpha}) \geq f(\infty)=\int_{-\infty}^{\infty} (z^2 - \alpha) \exp (-z^2/2) dz$. Rewriting $f(\sqrt{\alpha})$ as 
\begin{align*}
%\begin{split}
f(\sqrt{\alpha})  
%&= \int_{\frac{x - \rho \sqrt{\alpha}}{\sqrt{1 - \rho^2}}}^\infty (z^2 - \alpha) \exp \left(-\frac{z^2}{2}\right) dz + \int_{-\infty}^{-\frac{x + \rho \sqrt{\alpha}}{\sqrt{1 - \rho^2}}} (z^2 - \alpha) \exp \left(-\frac{z^2}{2}\right) dz \\
& = f(\infty) - \int_{-\frac{x + \rho \sqrt{\alpha}}{\sqrt{1 - \rho^2}}}^{\frac{x - \rho \sqrt{\alpha}}{\sqrt{1 - \rho^2}}} (z^2 - \alpha) \exp \left(-\frac{z^2}{2}\right) dz
%\end{split}
\end{align*}
we see that inequality $f(\infty) \leq f(\sqrt{\alpha})$ follows from  \eqref{eq:upper3_7}. This proves item (iii) and thus
\eqref{eq:upper3_6bis}. Therefore, the proof of \eqref{thirdsum} is complete. 
Combining \eqref{eq:upper3_1}, \eqref{eq:upper3_2}, \eqref{eq:upper3_3}, \eqref{secsum} and \eqref{thirdsum} yields the lemma.
\end{proof}

%%%%%%%%%%%%%%%%%%%%%%%%

\subsection{Proof of Theorem \ref{thm:upper2}}\label{sec:proof_upper}

We consider separately the cases $s \geq \sqrt{p}$ and $s < \sqrt{p}$.
\medskip

{\it Case $s \geq \sqrt{p}$.}
From~\eqref{1} we get that, almost surely,
\begin{equation*}
(X^TX)^{-1}X^TY = \theta + \tilde{\epsilon}
\end{equation*}
where $$\tilde{\epsilon}= \sigma (X^TX)^{-1}X^T\xi.$$
Thus, we have
 \begin{align}\label{up_1}
 \begin{split}
\esp_{\theta}\big[\big(\hat{Q} - \|\theta\|_2^2\big)^2\big]& = \esp_{\theta}\big(2 \theta^T\tilde{\epsilon} +  \|\tilde{\epsilon}\|_2^2- \sigma^2 \trace\big[(X^TX)^{-1}\big] \big)^2 \\
&\leq 8 \,\esp_{\theta}\big(\theta^T\tilde{\epsilon}\big)^2 + 2\,\esp_{\theta}\Big(\|\tilde{\epsilon}\|_2^2-\sigma^2 \trace\big[(X^TX)^{-1}\big]\Big)^2.
\end{split}
\end{align}
Note that, conditionally on $X$,  the random vector $\tilde{\epsilon}$ is normal with mean~0 and covariance matrix $\sigma^2(X^TX)^{-1}.$
Thus, conditionally on $X$, the random variable $\theta^T\tilde{\epsilon}$ is normal with mean~0 and variance $\sigma^2\theta^T(X^TX)^{-1}\theta$. It follows that $\esp_{\theta}\big(\theta^T\tilde{\epsilon}\big)^2\le \sigma^2r^2\esp\big[\lambda_{\min}^{-1}(X^TX)\big]$. Hence, applying Lemma~\ref{lem:upper1} we have, for some constant $C$ depending only on $\gamma$, 
\begin{equation}\label{up_2}
\esp_{\theta}\big( \theta^T\tilde{\epsilon})^2\le C\sigma^2\frac{r^2}{n}.
\end{equation}
 Consider now the second term on the right hand side of \eqref{up_1}. Denote by $(\lambda_i,u_i)$, $i=1,\ldots,p,$ the eigenvalues and the corresponding orthonormal eigenvectors of $(X^TX)^{-1}$, respectively.  Set   $v_i=\sqrt{\lambda_i}u_i^TX^T \xi$. We have 
 \begin{equation*}
 \esp_{\theta}\Big(\|\tilde{\epsilon}\|_2^2-\sigma^2\trace\big[(X^TX)^{-1}\big]\Big)^2 = \sigma^4 \esp \Big( \sum_{i=1}^p \lambda_i[v_i^2-1]\Big)^2.\\
\end{equation*}
Conditionnally on $X$, the random variables $v_1,\dots,v_p$ are i.i.d. standard Gaussian. Using this fact and Lemma~\ref{lem:upper1} we get that, for some constant $C$ depending only on $\gamma$, 
\begin{align}\label{up_3}
%\begin{split}
 \esp_{\theta}\Big(\|\tilde{\epsilon}\|_2^2-\sigma^2\trace\big[(X^TX)^{-1}\big]\Big)^2& = 2 \sigma^4 \esp\Big(\sum_{i=1}^p   \lambda_i^2 \Big)
 \leq 2 p \sigma^4 \esp \big[  \lambda_{\min}^{-2}\big(X^TX\big)\big]
 \le C \frac{\sigma^4 p}{n^2}.
%\end{split}
\end{align}
Combining \eqref{up_1},  \eqref{up_2} and  \eqref{up_3} we obtain the result of the theorem for $s \geq \sqrt{p}$. 

%%%%%%%%%%%%%%%%%%%

\medskip

{\it Case $s < \sqrt{p}$.}
Set $S = \{i : \theta_i \neq 0\}$.  We have
\begin{equation} \label{eq:upper1}
\begin{split}
\esp_{\theta}\big(\hat{Q} - \|\theta\|_2^2\big)^2 & \leq 3 \esp_{\theta} \Big( \underset{i \in S}{\sum} (y_i^2 - \sigma^2 (X^T X)_{ii}^{-1} \alpha_s - \theta_i^2) \Big)^2 \\
& + 3 \esp_{\theta} \Big( \underset{i \in S}{\sum} \left[ y_i^2 - \sigma^2 (X^T X)_{ii}^{-1} \alpha_s \right] \fcar_{y_i^2 \leq 2 \sigma^2 (X^T X)_{ii}^{-1} \log (1 + p / s^2)} \Big)^2 \\
& + 3 \esp_{\theta} \Big( \underset{i \not\in S}{\sum} \Big[ \tilde{\epsilon}_i^2 - \sigma^2 (X^T X)_{ii}^{-1} \alpha_s \Big] \fcar_{y_i^2 > 2 \sigma^2 (X^T X)_{ii}^{-1} \log (1 + p / s^2)}\Big)^2,
\end{split}
\end{equation}
where $\tilde{\epsilon}_i$ denotes the $i$th component of $\tilde{\epsilon}$. We now establish upper bounds for the three terms on the right hand side of~\eqref{eq:upper1}. 
For the first term, observe that
\begin{equation}\label{eq:upper2}
\esp_{\theta} \Big( \underset{i \in S}{\sum} (y_i^2 - \sigma^2 (X^T X)_{ii}^{-1} \alpha_s - \theta_i^2) \Big)^2 \leq 8\esp_{\theta}\Big( \underset{i \in S}{\sum} \theta_i \tilde{\epsilon}_i \Big)^2 + 2\esp_{\theta}\Big( \underset{i \in S}{\sum} (\tilde{\epsilon}_i^2 -\sigma^2 (X^TX)_{ii}^{-1} \alpha_s)  \Big)^2.
\end{equation}
The second summand on the right hand side of~\eqref{eq:upper2} satisfies 
\begin{align}\label{up_4}
\begin{split}
\esp_{\theta}\Big( \underset{i \in S}{\sum} (\tilde{\epsilon}_i^2 -\sigma^2 (X^TX)_{ii}^{-1} \alpha_s)  \Big)^2 &\leq 2 \sigma^4(\alpha_s^2 + 3) \esp \underset{i \in S}{\sum} \underset{j \in S}{\sum} (X^TX)_{ii}^{-1} (X^TX)_{jj}^{-1} \\
&\leq 2 \sigma^4(\alpha_s^2 + 3) s^2 \esp \left[ \lambda_{\min}^{-2}(X^TX) \right].
\end{split}
\end{align}
From~\eqref{lem:upper2_0_4} we obtain
\begin{equation}\label{alphas}
\alpha_s\le 10 \log (1 + p / s^2).
\end{equation}
Thus, using  \eqref{eq:upper2}, \eqref{up_4} and \eqref{alphas} together with Lemma~\ref{lem:upper1} and \eqref{up_2} we find 
\begin{equation}\label{up_5}
\esp_{\theta} \Big( \underset{i \in S}{\sum} (y_i^2 - \sigma^2 (X^T X)_{ii}^{-1} \alpha_s - \theta_i^2) \Big)^2 \leq C  \sigma^4 s^2 \log^2 (1 + p / s^2) / n^2,
\end{equation} 
where the constant $C$ depends only on $\gamma$. 
For the second term on the right hand side of~\eqref{eq:upper1}, we have immediately  that it is smaller, up to an absolute constant factor,  than
\begin{equation*}
\esp \sigma^4 \underset{i \in S}{\sum} \underset{j \in S}{\sum} (X^TX)_{ii}^{-1} (X^TX)_{jj}^{-1} (\alpha_s^2 +  4\log^2 (1 + p / s^2)).
\end{equation*}
Arguing as in \eqref{up_4} and applying Lemma~\ref{lem:upper1} and \eqref{alphas} we get that, for some constant $C$ depending only on $\gamma$, 
\begin{align}\label{up_6}
%\begin{split}
&\esp_{\theta} \Big( \underset{i \in S}{\sum} \left[ y_i^2 - \sigma^2 (X^T X)_{ii}^{-1} \alpha_s \right] \fcar_{y_i^2 \leq 2 \sigma^2 (X^T X)_{ii}^{-1} \log (1 + p / s^2)} \Big)^2 
\le C \sigma^4 s^2 \log^2 (1 + p / s^2) / n^2.
%\end{split}
\end{align}
For the third term on the right hand side of~\eqref{eq:upper1}, we have
\begin{align}\label{up_6a}
\begin{split}
& \esp_{\theta} \Big( \underset{i \not\in S}{\sum} \Big[ \tilde{\epsilon}_i^2 - \sigma^2 (X^T X)_{ii}^{-1} \alpha_s \Big] \fcar_{y_i^2 > 2 \sigma^2 (X^T X)_{ii}^{-1} \log (1 + p / s^2)}\Big)^2\\
%\sum_{i=1}^p \sum_{j=1}^p \esp \left( (\tilde{\epsilon}_i^2 - \sigma^2 \alpha_s (X^T X)_{ii}^{-1}) (\tilde{\epsilon}_j^2 - \sigma^2 \alpha_s (X^T X)_{jj}^{-1}) \fcar_{\tilde{\epsilon}_i^2 > 2 \sigma^2 (X^T X)_{ii}^{-1} \log(1 + p / s^2)} \fcar_{\tilde{\epsilon}_j^2 > 2 \sigma^2 (X^T X)_{jj}^{-1} \log(1 + p / s^2)} \right) \\
&\qquad = \sigma^4 \sum_{i \not\in S} \sum_{j \not\in S} \esp \left( (X^T X)_{ii}^{-1} (X^T X)_{jj}^{-1} (\tilde{\xi}_i^2 - \alpha_s) (\tilde{\xi}_j^2 - \alpha_s) \fcar_{\vert \tilde{\xi}_i \vert > x} \fcar_{\vert \tilde{\xi}_j \vert > x}\right),
\end{split}
\end{align}
where 
$$
x = \sqrt{2 \log(1 + p / s^2)},\quad \tilde{\xi}_i = \frac{\tilde{\epsilon}_i }{\sqrt{\sigma^2(X^TX)^{-1}_{ii}}}.
$$  
Note that $\esp (\tilde{\xi}_i^2\vert X)=\esp(\tilde{\xi}_j^2\vert X)=1$ and, conditionally on $X$,  $(\tilde{\xi}_i,\tilde{\xi}_j)\in\R^2$ is a centered Gaussian vector with   covariance  
$$\rho_{ij} =  \frac{(X^T X)_{ij}^{-1}}{\sqrt{(X^T X)_{ii}^{-1}} \sqrt{(X^T X)_{jj}^{-1}}}.$$
Using Lemma~\ref{lem:upper3} we obtain that, for some absolute positive constants $C$, 
\begin{align*}
\begin{split}
& \sum_{i \not\in S} \sum_{j \not\in S}\esp \left( (X^T X)_{ii}^{-1} (X^T X)_{jj}^{-1} (\tilde{\xi}_i^2 - \alpha_s) (\tilde{\xi}_j^2 - \alpha_s) \fcar_{\vert \tilde{\xi}_i \vert > x} \fcar_{\vert \tilde{\xi}_j \vert > x}\right) \\
&= \sum_{i \not\in S} \sum_{j \not\in S} \esp \left( (X^T X)_{ii}^{-1} (X^T X)_{jj}^{-1}\esp\Big[ (\tilde{\xi}_i^2 - \alpha_s) (\tilde{\xi}_j^2 - \alpha_s) \fcar_{\vert \tilde{\xi}_i \vert > x} \fcar_{\vert \tilde{\xi}_j \vert > x}\mid X\Big]\right) 
\\
&  \leq C \sum_{i,j=1}^p \esp \left[ (X^T X)_{ii}^{-1} (X^T X)_{jj}^{-1} \rho_{ij}^2 \right] x^4 \exp(-{x^2}/{2})
\\
& = C  \esp \left[  \| (X^T X)^{-1}  \|_F^2\right] x^4 \exp(-{x^2}/{2})
\\
&\le  C  \esp \left[  \| (X^T X)^{-1}  \|_F^2\right]  \frac{s^2}{p} \log^2(1 + p / s^2) \\
&\le C \esp\left[ \lambda_{\min}^{-2}(X^T X)\right] s^2 \log^2(1 + p / s^2),
\end{split}
\end{align*}
where $ \| (X^T X)^{-1}  \|_F$ is the Frobenius norm of matrix $(X^T X)^{-1}$.
Finally, Lemma~\ref{lem:upper1}, \eqref{up_6a} and the last display imply that, for some constant $C$ depending only on $\gamma$, 
 \begin{align}\label{up_7}
 %\begin{split}
 & \esp_{\theta} \Big( \underset{i \not\in S}{\sum} \Big[ \tilde{\epsilon}_i^2 - \sigma^2 (X^T X)_{ii}^{-1} \alpha_s \Big] \fcar_{y_i^2 > 2 \sigma^2 (X^T X)_{ii}^{-1} \log (1 + p / s^2)}\Big)^2
% \sum_{i=1}^p \sum_{j=1}^p \esp \left( (\tilde{\epsilon}_i^2 - \sigma^2 \alpha_s (X^T X)_{ii}^{-1}) (\tilde{\epsilon}_j^2 - \sigma^2 \alpha_s (X^T X)_{jj}^{-1}) \fcar_{\tilde{\epsilon}_i^2 > 2 \sigma^2 (X^T X)_{ii}^{-1} \log(1 + p / s^2)} \fcar_{\tilde{\epsilon}_j^2 > 2 \sigma^2 (X^T X)_{jj}^{-1} \log(1 + p / s^2)} \right) \\
\le  C \frac{\sigma^4 s^2 \log^2(1 + p / s^2)}{n^2}.
 %\end{split}
 \end{align}
 The proof is completed by combining \eqref{eq:upper1},  \eqref{up_5},  \eqref{up_6} and   \eqref{up_7}.

%%%%%%%%%%%%%%%%%%%%%%%

\subsection{Preliminary lemmas for the proof of Theorem \ref{th:lower}}

We first recall some general facts about lower bounds for the risks of tests. Let $\Theta$ be a measurable set, not necessarily the set $\Theta(s,\tau)$,
and let $\mu$ be a probability measure on $\Theta$. Consider any family of probability measures $\prob_\theta$ indexed by $\theta\in \Theta$. Denote by ${\mathbb P}_\mu$ the mixture probability measure  
\begin{equation*}
{\mathbb P}_\mu = \int_\Theta \prob_\theta\,\mu(d\theta).
\end{equation*}
Let $$\chi^2(P',P)=\int (dP'/dP)^2 dP -1$$ 
be the chi-square divergence between two probability measures $P'$ and~$P$ if $P'\ll P$, and set $\chi^2(P',P)=+\infty$ otherwise. The following lemma is a key tool in Le Cam's method of proving lower bounds (see, e.g., \cite[Lemma 3]{CCT2017}).
\begin{lemma}\label{lem:lower1}
Let $\mu$ be a probability measure on $\Theta$,  and let $\{\prob_\theta: \theta\in \Theta\}$ be a family of probability measures indexed by $\theta\in \Theta$ on $\mathcal X$. Then, for any probability measure $Q$ on  $\mathcal X$,
$$
 \inf_\Delta \Big\{Q(\Delta=1) + \sup_{\theta\in \Theta}\prob_\theta(\Delta=0)\Big\} \ge 1-\sqrt{\chi^2({\mathbb P}_{\mu} ,Q) }
$$
where $\inf_\Delta$ is the infimum over all $\{0,1\}$-valued statistics. 
\end{lemma} 

Applying Lemma \ref{lem:lower1} with $Q=\prob_0$, we see that it suffices to choose a suitable measure $\mu$ and to bound $\chi^2({\mathbb P}_{\mu} ,\prob_0)$ from above by a small enough value in order to obtain the desired lower bound on ${\mathcal R}_{s,\tau}$. The following lemma is useful to evaluate $\chi^2({\mathbb P}_{\mu} ,\prob_0)$.

\begin{lemma}\label{lem:lower2}
Let $\mu$ be a probability measure on $\Theta$,  and let $\{\prob_\theta: \theta\in \Theta\}$ be a family of probability measures indexed by $\theta\in \Theta$ on $\mathcal X$. Let $Q$ be a probability measure on  $\mathcal X$ such that $\prob_\theta\ll Q$ for all $\theta\in \Theta$. Then, 
$$
\chi^2({\mathbb P}_{\mu} ,Q) = \mathbb E_{(\theta, \theta') \sim   \mu^2}\Big(\int \frac{d\prob_\theta d\prob_{\theta'}}{dQ}\Big) -1.
 $$
 Here, $\mathbb E_{(\theta, \theta' )\sim   \mu^2}$ denotes the expectation with respect to the distribution of the pair $(\theta, \theta')$ where $\theta$ and $\theta'$ are independent and each of them is distributed according to $\mu$.
\end{lemma} 
\begin{proof} 
It suffices to note that
$$
\chi^2({\mathbb P}_{\mu} ,Q) = \int \frac{(d{\mathbb P}_{\mu})^2}{dQ} - 1
$$
whereas
\begin{eqnarray*}
\int \frac{(d{\mathbb P}_{\mu})^2}{dQ} &=& \int \frac{\int_\Theta d\prob_\theta \mu(d\theta) \int_\Theta d\prob_{\theta'}\mu(d\theta')}{dQ}= \int_\Theta \int_\Theta \mu(d\theta)\mu(d\theta') \int \frac{d\prob_\theta d\prob_{\theta'}}{dQ}.
\end{eqnarray*}
\end{proof}

We now specify the expression for the $\chi^2$ divergence in Lemma \ref{lem:lower2} when $\prob_\theta$ is the probability distribution generated by model \eqref{1} and $Q=\prob_0$.

\begin{lemma}\label{lem:lower3}
Let $\prob_\theta$ be the distribution of $(X,Y)$ satisfying \eqref{1}. Then,
$$
\chi^2({\mathbb P}_{\mu} ,\prob_0) = \mathbb E_{(\theta, \theta') \sim   \mu^2} E_X \exp(\langle X\theta, X\theta'\rangle/\sigma^2) -1.
 $$
\end{lemma} 
\begin{proof} We apply Lemma \ref{lem:lower2} and notice that, for any $(\theta, \theta')\in \Theta \times \Theta$,  
\begin{eqnarray*}
\int \frac{d\prob_\theta d\prob_{\theta'}}{d\prob_0} &=&
\frac{1}{(2\pi\sigma)^{n/2}}E_X\int_{\R^n} \exp\Big(-\frac{1}{2\sigma^2}(\|y-X\theta\|_2^2 + \|y-X\theta'\|_2^2 - \|y\|_2^2)\Big)dy\\
&=&
\frac{1}{(2\pi\sigma)^{n/2}}E_X \int_{\R^n} \exp\Big(-\frac{1}{2\sigma^2}(\|y\|_2^2 - 2\langle y,X(\theta+\theta')\rangle + \|X(\theta+\theta')\|_2^2 - 2\langle X\theta, X\theta'\rangle)\Big)dy\\
&=&
E_X \left( \frac{\exp(\langle X\theta, X\theta'\rangle/\sigma^2)}{(2\pi\sigma)^{n/2}}\int_{\R^n} \exp\Big(-\frac{1}{2\sigma^2}\|y-X(\theta+\theta')\|_2^2  \Big)dy\right)
\\
&=& E_X \exp(\langle X\theta, X\theta'\rangle/\sigma^2).
\end{eqnarray*}
\end{proof}

\begin{lemma}\label{lem:lower4}
Let $a \in \mathbb R$ be a constant and let $W$ be a random variable. Then,
$$
\mathbf E \exp(W) \leq \exp(a) \big(1+  \int_0^{\infty}  e^{t} p(t) dt\big)
$$
where $p(t) =\mathbf P \big(|W - a| \geq t \big)$. 
\end{lemma}
\begin{proof}
We have
\begin{align*}
\mathbf E \exp(W) &\leq \exp(a) \mathbf E \exp(|W-a|)\\
&= \exp(a)  \int_{0}^{\infty} \mathbf P \big(\exp(|W-a|) \geq x\big) dx\\
&= \exp(a) \Big[ 1 + \int_1^{\infty}  \mathbf P \big(\exp(|W-a|) \geq x \big)  dx \Big]\\
&= \exp(a) \Big[ 1 + \int_0^{\infty} e^{t}  p(t)dt \Big].
\end{align*}
\end{proof}
%%%%%%%%%%%%%%%%%%%%%
\begin{lemma}\label{lem:lower5}
Assume that matrix $X$ has an isotropic distribution with independent $\sigma_X$-subgaussian rows for some  $\sigma_X>0$. Then, for all $x>0$ and all $\theta, \theta'\in \R^p$ we have
$$
P_X \Big(|\langle X\theta, X\theta'\rangle- n\langle \theta, \theta'\rangle| \ge  \|\theta\|_2 \|\theta'\|_2 x\Big)\le 
6 \exp(-C_1\min(x, x^2/n))
$$
where the constant $C_1>0$ depends only on $\sigma_X$.
\end{lemma}
\begin{proof}
By homogeneity, it is enough to consider the case $\|\theta\|_2 = \|\theta'\|_2=1$, which will be assumed in the rest of the proof. Then we have
$$ 
\langle X\theta, X\theta'\rangle= \frac{\|X\theta\|_2^2 + \|X\theta'\|_2^2 - \|X(\theta-\theta')\|_2^2}{2},
\qquad
\langle \theta, \theta'\rangle =  \frac{2 -  \|\theta-\theta'\|_2^2}{2},$$
which implies 
\begin{align}
\big|\frac1n \langle X\theta, X\theta'\rangle - \langle \theta, \theta'\rangle \big|\leq \frac{1}{2}\Big(&\Big|\frac1n\|X\theta\|_2^2 - 1\Big| + \Big|\frac1n\|X\theta'\|_2^2- 1\Big |\nonumber\\ 
&+  \Big|\frac1n\|X(\theta-\theta')\|_2^2 - \|\theta-\theta'\|_2^2\Big| \Big).\label{eq1:lemma5}
\end{align}
By renormalizing, the third summand on the right hand side of \eqref{eq1:lemma5} is reduced to the same form as the first two summands. Thus, to prove the lemma it suffices to show that 
\begin{align}
P_X \Big(\Big|\frac1n\|X\theta\|_2^2 - 1 \Big|\ge v
\Big)\le 
2 \exp(-C_1'\min(v, v^2)n), \quad \forall \ v>0, \|\theta\|_2=1,
\label{eq2:lemma5}
\end{align}
where the constant $C_1'>0$ depends only on $\sigma_X$.

Denote by ${\bf x}_i$ the $i$th row of matrix $X$. Then 
$$
\frac1n\|X\theta\|_2^2 - 1= \frac1n \sum_{i=1}^n(Z_i^2-1),
$$ 
where $Z_i={\bf x}_i^T\theta$ are independent  $\sigma_X$-subgaussian random variables, such that $\esp(Z_i^2)=1$ for $i=1,\dots,n$. Therefore,  $Z_i^2-1$, $i=1,\dots,n$, are independent centered sub-exponential random variables and 
\eqref{eq2:lemma5} immediately follows from Bernstein's inequality for sub-exponential random variables (cf., e.g.,\cite{versh}, Corollary 5.17).
\end{proof}

%%%%%%%%%%%%%%%%%%%%
\begin{lemma}\label{lem:lower6}
Assume that matrix $X$ has an isotropic distribution with independent $\sigma_X$-subgaussian rows for some  $\sigma_X>0$. Then, there exists $u_0>0$ depending only on $\sigma_X$ such that, for all $\theta, \theta'$ with $\|\theta\|_2, \|\theta'\|_2\le un^{-1/4}$ and $u\in (0,u_0)$ we have 
$$
E_X \exp(\langle X\theta, X\theta'\rangle)\le \exp(n\langle \theta, \theta'\rangle) (1+ C_0u^2)
$$
where the constant $C_0>0$ depends only on $\sigma_X$.
\end{lemma}
\begin{proof}
By Lemma~\ref{lem:lower5}, for any  $x>0$ with $P_X$-probability at least $1-6e^{-C_1\min(x, x^2/n)}$ we have
$$\Big|\langle  X\theta, X \theta'\rangle - n\langle \theta, \theta'\rangle \Big| \leq \|\theta\|_2 \|\theta'\|_2 x \leq u^2 n^{-1/2}x.$$
Therefore, for any  $t>0$ with $P_X$-probability at least $1-6e^{-C_1\min(\sqrt{n}t/u^2,t^2/u^4)}$ we have
$$\Big|\langle  X\theta,  X\theta'\rangle - n\langle \theta, \theta'\rangle \Big|\le t.
$$
This and Lemma~\ref{lem:lower4} imply that, for all $u\le u_0:=(C_1/2)^{1/2}$,
\begin{align}
&E_X \exp(\langle X \theta,X\theta'\rangle) \leq \exp(n\langle \theta, \theta'\rangle)
\Big(1 +  6\int_{0}^{\infty}  e^{t-C_1\min\big(\sqrt{n}t/u^2,t^2/u^4)\big)} dt\Big) \nonumber
\\
&\leq \exp(n\langle \theta, \theta'\rangle)\Big(1 + 6\int_{0}^{\infty} e^{t(1-C_1\sqrt{n}/u^2)} dt + 6\int_{0}^{\infty} e^{t-C_1t^2/u^4} dt\Big) \nonumber
\\
&\leq \exp(n\langle \theta, \theta'\rangle) \Big(1 + 6\int_{0}^{\infty} e^{-C_1\sqrt{n} t/(2u^2)} dt + 6\int_{0}^{\infty} e^{-t(C_1t/u^4 - 1)} dt\Big) \quad \mbox{(as \ $C_1\sqrt{n}/u^2>2$)} \nonumber
\\
&\leq \exp(n\langle \theta, \theta'\rangle) \Big(1 + \frac{12u^2}{C_1\sqrt{n}} +  \frac{12u^4}{C_1}e^{2u^4/C_1} + 6\int_{2u^4/C_1}^{\infty} e^{-t^2C_1/(2u^4)} dt\Big) \nonumber
\\
%&\leq \exp(n\langle \theta, \theta'\rangle) \Big(1 + 6u^2/C  + \sqrt{2}u^2/\sqrt{C})\Big) \nonumber\\
&\leq \exp(n\langle \theta, \theta'\rangle) \big(1 +  C_0u^2 \big),\label{eq:rregt4}
\end{align}
where the constant $C_0>0$ depends only on $C_1$, and thus only on $\sigma_X$.
\end{proof}

\subsection{Proof of Theorem \ref{th:lower}}\label{sec:proof_lower}

For an integer $s$ such that $1\le s \le p$ and $\tau>0$, we denote by $\mu_\tau$ the uniform distribution on the set of vectors in $\R^p$ having exactly $s$ nonzero coefficients, all equal to $\tau/\sqrt{s}$. Note that the support of measure $\mu_\tau$ is contained in $\Theta(s,\tau)$. 

We now take $\tau=\tau(s,n,p)$ defined by \eqref{eq:th:lower} and set $\mu= \mu_\tau$. In view of Lemmas~\ref{lem:lower1} - \ref{lem:lower3}, to prove Theorem~\ref{th:lower} it is enough to show that 
\begin{equation}\label{eq1:proof:th1}
\mathbb E_{(\theta, \theta') \sim   \mu_\tau^2}E_X \exp(\langle X\theta, X\theta'\rangle/\sigma^2) \le 1+o_A(1)
\end{equation}
where $o_A(1)$ tends to 0 as $A\to 0$.

Before proving \eqref{eq1:proof:th1} we proceed to some simplifications. First, note that for $\tau$ defined by \eqref{eq:th:lower} the left hand side of \eqref{eq1:proof:th1} does not depend on $\sigma$. Thus,  in what follows we set $\sigma=1$ without loss of generality. Next, notice that it is enough to prove the theorem for the case $s\le \sqrt{p}$. Indeed, for $s> \sqrt{p}$ we can use the inclusions $\Theta(s,\tau(s,n,p))\supseteq \Theta(s',\tau(s,n,p))\supseteq \Theta(s',\tau(s',n,p))$ where $s'$ is the greatest integer smaller than or equal to $\sqrt{p}$. Since
$$
\tau(s',n,p)\asymp \min  \Big( \frac{p^{1/4}}{\sqrt{n}}, n^{-1/4}  \Big)
$$  
and the rate \eqref{eq:th:lower} is also of this order for $s> \sqrt{p}$, it suffices to prove the lower bound for $s\le s'$, and thus for $s\le \sqrt{p}$. Taking onto account these simplifications,  we assume in what follows without loss of generality that $s\le  \sqrt{p}$, $\sigma=1$, and 
\begin{equation}\label{eq2:proof:th1}
\tau : = A\min  \Big(\sqrt{\frac{s\log(1 + p/s^2)}{n}}, n^{-1/4} \Big). 
\end{equation}
We now prove \eqref{eq1:proof:th1} under these assumptions. By Lemma~\ref{lem:lower6}, for any $0<A<u_0$ we have 
\begin{equation}\label{eq3:proof:th1}
\mathbb E_{(\theta, \theta') \sim   \mu_\tau^2}E_X \exp(\langle X\theta, X\theta'\rangle) \le \mathbb E_{(\theta, \theta') \sim   \mu_\tau^2}\exp(n\langle \theta, \theta'\rangle)(1+C_0A^2).
\end{equation}
Assume that $A<1$. Arguing exactly as in the proof of Lemma 1 in \cite{CCT2017}, we find
\begin{eqnarray}\label{eq4:proof:th1}
\mathbb E_{(\theta, \theta') \sim   \mu_\tau^2}\exp(n\langle \theta, \theta'\rangle)&= &\mathbb E_{(\theta, \theta') \sim   \mu_\tau^2}\exp\Big(n\tau^2 s^{-1} \sum_{j=1}^p \fcar_{\theta_j\ne 0} \fcar_{\theta'_j\ne 0} \Big)\\
&\le & \Big(1-\frac{s}{p} + \frac{s}{p}\exp(n\tau^2 s^{-1}) \Big)^s\nonumber\\
&\le & \Big(1-\frac{s}{p} + \frac{s}{p} \left(1+\frac{p}{s^2}\right)^{A^{2}}\Big)^s\nonumber\\
&\le& \Big(1 +\frac{A^2}{s}\Big)^s  \le \exp(A^2) \nonumber
\end{eqnarray}
where we have used the inequality $(1+x)^{A^{2}}-1\le A^2 x$ valid for $0<A<1$ and $x>0$. Combining \eqref{eq3:proof:th1} and \eqref{eq4:proof:th1} we obtain that, for all $0<A<\min(1, u_0)$, 
\begin{equation*}\label{eq5:proof:th1}
\mathbb E_{(\theta, \theta') \sim   \mu_\tau^2}E_X \exp(\langle X\theta, X\theta'\rangle) \le \exp(A^2)(1+C_0A^2)
\end{equation*}
with some $u_0>0$ and $C_0>0$ depending only on $\sigma_X$. This completes the proof of the theorem. 

\section*{Acknowledgements}

The work of Alexandra Carpentier is supported by the Emmy Noether grant MuSyAD CA 1488/1-1, by the GK 2297 MathCoRe on “Mathematical Complexity Reduction"  – 314838170, GRK 2297 MathCoRe, and by the SFB 1294 Data Assimilation on “Data Assimilation — The seamless integration of data and models", Project A03, all funded by the Deutsche Foschungsgemeinschaft (DFG, German Research Foundation), and by the Deutsch-Französisches Doktorandenkolleg/ Collège doctoral franco-allemand Potsdam-Toulouse CDFA 01-18, funded by the French-German University. Olivier Collier's research has been conducted as part of the project Labex MME-DII (ANR11-LBX-0023-01). The work of A.B.Tsybakov was supported by GENES and by the French National Research Agency (ANR) under the grants IPANEMA (ANR-13-BSH1-0004-02) and Labex Ecodec (ANR-11-LABEX-0047).

\end{document}